\documentclass[11pt, a4paper]{article}
\usepackage{picinpar}
\usepackage{epsfig}
\usepackage{mathrsfs}
\usepackage{amsmath}
\usepackage{amsthm,amsfonts}
\usepackage[all]{xy}
\usepackage{graphicx}
\usepackage{indentfirst}
\usepackage{amssymb}
\usepackage{float}
\newtheorem{theorem}{Theorem}[section]
\newtheorem{lemma}{Lemma}[section]
\newtheorem{corollary}{Corollary}[section]
\newtheorem{conjecture}{Conjecture}[section]

\begin{document}

\begin{center}

{\huge\bf On uniquely 3-colorable plane graphs without prescribed adjacent faces
\footnote{This research is supported by 973 Program of China
 (Grant Nos: 2013CB329600)}}\\[6pt]
\end{center}

\begin{center}
\centerline{\large{ Ze-peng LI}}
{\footnotesize{School of Electronics Engineering and Computer Science}}\\
  {\footnotesize{Key Laboratory of High Confidence Software}} \\
  {\footnotesize{Technologies of Ministry of Education,}}\\
  {\footnotesize{ Peking University, Beijing 100871, China }} \\
\hspace{0.4cm}

\centerline{\large{ Naoki MATSUMOTO}}
{\footnotesize{Graduate School of Environment and Information}}\\
  {\footnotesize{Sciences Yokohama National University,}} \\
  {\footnotesize{ Yokohama, Japan}} \\
\hspace{0.4cm}

\centerline{\large{ En-qiang ZHU}}
{\footnotesize{School of Electronics Engineering and Computer Science}}\\
  {\footnotesize{Key Laboratory of High Confidence Software}} \\
  {\footnotesize{Technologies of Ministry of Education,}}\\
  {\footnotesize{ Peking University, Beijing 100871, China }} \\
\hspace{0.4cm}

\centerline{\large{ Jin XU}}
{\footnotesize{School of Electronics Engineering and Computer Science}}\\
  {\footnotesize{Key Laboratory of High Confidence Software}} \\
  {\footnotesize{Technologies of Ministry of Education,}}\\
  {\footnotesize{ Peking University, Beijing 100871, China }} \\
\hspace{0.4cm}

\centerline{\large{Tommy JENSEN}}
{\footnotesize{Kyungpook National University  }}\\
  {\footnotesize{ 1370 Sankyuk-dong Buk-gu Daegu, 702701, Korea }} \\
\hspace{0.4cm}

\end{center}

\begin{abstract}
{\footnotesize  A graph $G$ is \emph{uniquely k-colorable} if the chromatic number of $G$ is $k$ and $G$ has
only one $k$-coloring up to permutation of the colors. For a plane graph $G$, two faces $f_1$ and $f_2$ of $G$ are \emph{adjacent $(i,j)$-faces}
if $d(f_1)=i$, $d(f_2)=j$ and $f_1$ and $f_2$ have a common edge,
where $d(f)$ is the degree of a face $f$.
In this paper, we prove that every uniquely 3-colorable plane graph has adjacent $(3,k)$-faces, where $k\leq 5$.
The bound 5 for $k$ is best possible. Furthermore, we prove that there exist a class of uniquely 3-colorable plane graphs having neither adjacent $(3,i)$-faces nor adjacent $(3,j)$-faces, where $i,j\in \{3,4,5\}$ and $i \neq j$.
One of our constructions implies that there exist an infinite family of edge-critical uniquely 3-colorable plane graphs with $n$ vertices and $\frac{7}{3}n-\frac{14}{3}$ edges, where $n(\geq 11)$ is odd and $n\equiv 2\pmod{3}$.\\[3pt]
\textbf{MR(2010)} Subject classification: 05C15 }
\end{abstract}

\section{Introduction}

   {\small
For a plane graph $G$, $V(G)$, $E(G)$ and $F(G)$ are the sets of vertices, edges and faces of $G$, respectively. The degree of a vertex $v \in V(G)$, denoted by $d_{G}(v)$, is the number of neighbors of $v$ in $G$. The degree of a face $f\in F(G)$, denoted by $d_G(f)$, is the number of edges in its boundary, cut edges being counted twice. When no confusion can arise, $d_{G}(v)$ and $d_{G}(f)$ are simplified by $d(v)$ and $d(f)$, respectively. A face $f$ is a \emph{$k$-face} if $d(f)=k$ and a \emph{$k^+$-face} if $d(f)\geq k$.  Two faces $f_1$ and $f_2$ of $G$ are \emph{adjacent $(i,j)$-faces} if $d(f_1)=i$, $d(f_2)=j$ and $f_1$ and $f_2$ have at least one common edge. Two distinct paths of $G$ are \emph{internally disjoint} if they have no internal vertices in common.

 A graph $G$ is \emph{uniquely k-colorable} if $\chi(G)=k$ and $G$ has only one $k$-coloring
up to permutation of the colors, where the coloring is called a \emph{unique $k$-coloring} of $G$. In other words, all $k$-colorings
of $G$ induce the same partition of $V(G)$ into $k$ independent sets, in which an independent set is called a \emph{color class} of $G$. In addition, uniquely colorable graphs may be defined in terms of their chromatic polynomials, which initiated by Birkhoff \cite{Birkhoff1912} for planar graphs in 1912, and for general graphs by Whitney \cite{Whitney1932} in 1932. Because a graph $G$ is uniquely $k$-colorable if and only if its chromatic polynomial is $k!$. For a discussion of chromatic polynomials, see Read \cite{Read1968}.

Uniquely colorable graphs were first studied by Harary and Cartwright \cite{Harary1968} in 1968. They proved the following theorem.

\begin{theorem}(Harary and Cartwright \cite{Harary1968}) \label{theorem1.1}
Let $G$ be a uniquely $k$-colorable graph. Then for any unique $k$-coloring of $G$, the subgraph induced by the union of any two color classes is connected.
\end{theorem}

As a corollary of Theorem \ref{theorem1.1}, it can be seen that a uniquely $k$-colorable graph $G$ has at least $(k-1)|V(G)|-{k \choose 2}$ edges.  There are many references on uniquely colorable graphs \cite{Chartrand1969,Harary1969,Bollob¨¢s1978}.

Chartrand and Geller \cite{Chartrand1969} in 1969 started to study uniquely colorable planar graphs. They proved that uniquely 3-colorable planar graphs with at least 4 vertices contain at least two triangles, uniquely 4-colorable planar graphs are maximal planar graphs, and uniquely 5-colorable planar graphs do not exist. Aksionov \cite{Aksionov1977} in 1977 improved the lower bound for the number of triangles in a uniquely 3-colorable planar graph. He proved that a uniquely 3-colorable planar graph with at least 5 vertices contains at least 3 triangles and gave a complete description of uniquely 3-colorable planar graphs containing exactly 3 triangles.

Let $G$ be a uniquely $k$-colorable graph, $G$ is \emph{edge-critical} if $G-e$ is not uniquely $k$-colorable for any edge $e\in E(G)$. Obviously, if a uniquely $k$-colorable graph $G$ has exactly $(k-1)|V(G)|-{k \choose 2}$ edges, then $G$ is edge-critical. Mel'nikov and Steinberg \cite{Mel'nikov1977} in 1977 asked to
find an exact upper bound for the number of edges in an edge-critical uniquely 3-colorable planar graph with $n$ vertices.
Recently, Matsumoto \cite{Matsumoto2013} proved that an edge-critical uniquely 3-colorable planar graph has at most $\frac{8}{3}n-\frac{17}{3}$ edges and  constructed an infinite family of edge-critical uniquely 3-colorable planar graphs with $n$ vertices and $\frac{9}{4}n-6$ edges, where $n\equiv 0 (\textrm{mod}~4)$.
\vspace{0.1cm}

In this paper, we mainly prove Theorem \ref{theorem1.2}.

 \begin{theorem} \label{theorem1.2}
If $G$ is a uniquely 3-colorable plane graph, then $G$ has adjacent $(3,k)$-faces, where $k\leq 5$. The bound 5 for $k$ is best possible.
\end{theorem}

Furthermore, by using constructions, we prove that there exist uniquely 3-colorable plane graphs having neither adjacent $(3,i)$-faces nor adjacent $(3,j)$-faces, where $i,j\in \{3,4,5\}$ and $i \neq j$. One of our constructions implies that there exist an infinite family of edge-critical uniquely 3-colorable plane graphs with $n$ vertices and $\frac{7}{3}n-\frac{14}{3}$ edges, where $n(\geq 11)$ is odd and $n\equiv 2\pmod{3}$.

}

\section{Proof of Theorem \ref{theorem1.2}}
\begin{lemma} \label{lemma2.1}
Let $G$ be a plane graph with 3-faces. If $G$ has no adjacent $(3,k)$-faces, where $k\leq 5$, then $|E(G)|\geq 2|F(G)|$.
\end{lemma}
\begin{proof}
We prove this by using a simple charging scheme. Since $G$ has no adjacent $(3,k)$-faces when $k\leq 5$, for any edge $e$ incident to a 3-face $f$, $e$ is incident to a face of degree at least 6.  Let $ch(f)=d(f)$ for any face $f\in F(G)$ and we call $ch(f)$ the initial charge of the face $f$. Let initial charges in $G$ be redistributed according to the following rule.

{\textbf{Rule}}: For each 3-face $f$ of $G$ and each edge $e$ incident with $f$, the $6^+$-face incident with $e$ sends $\frac{1}{3}$ charge to $f$ through $e$.

Denote by $ch^\prime(f)$ the charge of a face $f\in F(G)$ after applying redistributed Rule. Then

$$\sum_{f\in F(G)}ch^\prime(f)=\sum_{f\in F(G)}ch(f)=\sum_{f\in F(G)}d(f)=2|E(G)| \eqno{(1)}$$

On the other hand, for any 3-face $f$ of $G$, since the degree of each face adjacent to $f$ is at least 6, then by the redistributed Rule, $ch^\prime(f)=ch(f)+3\cdot\frac{1}{3}=d(f)+1=4$. For any 4-face or 5-face $f$ of $G$, $ch^\prime(f)=ch(f)=d(f)\geq 4$. For any $6^+$-face $f$ of $G$, since $f$ is incident to at most $d(f)$ edges each of which is incident to a 3-face, then $ch^\prime(f)\geq ch(f)-\frac{1}{3}d(f)=\frac{2}{3}d(f)\geq 4$. Therefore, we have
$$\sum_{f\in F(G)}ch^\prime(f)\geq \sum_{f\in F(G)}4=4|F(G)| \eqno{(2)}$$

By the formulae (1) and (2), we have $|E(G)|\geq 2|F(G)|$.
\end{proof}

\emph{Proof of Theorem \ref{theorem1.2}}
Suppose that the theorem is not true and let $G$ be a counterexample to the theorem. Then $G$ has at least one 3-face and no adjacent $(3,k)$-faces, where $k\leq 5$.
By Lemma \ref{lemma2.1}, $|E(G)|\geq 2|F(G)|$. Using Euler's Formula $|V(G)|-|E(G)|+|F(G)|=2$, we can obtain
$$|E(G)|\leq 2|V(G)|-4.$$

Since $G$ is uniquely 3-colorable, then by Theorem \ref{theorem1.1}, we have $|E(G)|\geq 2|V(G)|-3.$ This is a contradiction.

Note that the graph shown in Fig. \ref{figure1} is a uniquely 3-colorable plane graph having neither adjacent $(3,3)$-faces nor adjacent $(3,4)$-faces. Therefore, the bound 5 for $k$ is best possible. ~~~~~~~~~~~~~~~~~~~$\Box$

\begin{figure}[H]
\begin{center}
  \centering
  \includegraphics[width=120pt]{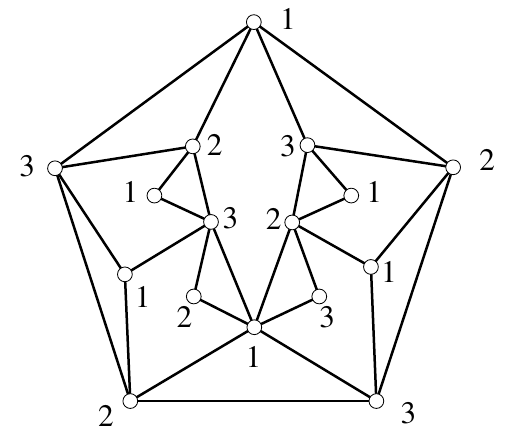}
\end{center}
\caption{A uniquely 3-colorable plane graph having neither adjacent $(3,3)$-faces nor adjacent $(3,4)$-faces}\label{figure1}
\end{figure}

\textbf{Remark}. By piecing together more copies of the plane graph in Fig. \ref{figure1}, one can construct an infinite class of uniquely 3-colorable plane graphs having neither adjacent $(3,3)$-faces nor adjacent $(3,4)$-faces.

\section{Construction of uniquely 3-colorable plane graphs without adjacent $(3,3)$-faces or adjacent $(3,5)$-faces}

There are many classes of uniquely 3-colorable plane graphs having neither adjacent $(3,4)$-faces nor adjacent $(3,5)$-faces, such as even maximal plane graphs (maximal plane graphs in which each vertex has even degree) and maximal outerplanar graphs with at least 6 vertices.

In this section, we construct a class of uniquely 3-colorable plane graphs having neither adjacent $(3,3)$-faces nor adjacent $(3,5)$-faces and prove that these graphs are edge-critical.

We construct a graph $G_k$ as follows: \\
(1) $V(G_k)=\{u,w,v_0,v_1,\ldots,v_{3k-1}\}$;\\
(2) $E(G_k)=\{v_0v_1, v_1v_2,\ldots,v_{3k-2}v_{3k-1},v_{3k-1}v_0\}\cup \{uv_i:i\equiv 1$ or $2 \pmod{3}\}\cup \{wv_i:i\equiv 0$ or $1 \pmod{3}\}$,
where $k$ is odd and $k\geq 3$. (See an example $G_3$ shown in Fig. \ref{figure2}.)

\begin{figure}[H]
\begin{center}
  \centering
  \includegraphics[width=180pt]{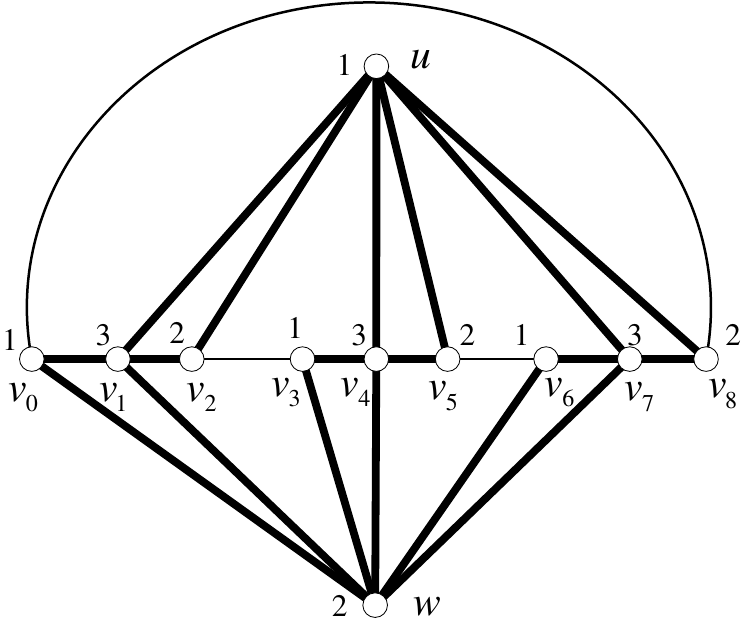}
\end{center}
\caption{An example $G_3$}\label{figure2}
\end{figure}

\begin{theorem} \label{theorem3.1}
For any odd $k$ with $k\geq 3$, $G_k$ is uniquely 3-colorable.
\end{theorem}
\begin{proof}
Let $f$ be any 3-coloring of $G_k$. Since $v_0v_1\ldots v_{3k-1}v_0$ is a cycle of odd length and each $v_i$ is adjacent to $u$ or $w$, we have $f(u)\neq f(w)$. Without loss of generality, let $f(u)=1$ and $f(w)=2$. By the construction of $G_k$, we know that $v_{3j+1}$ is adjacent to both $u$ and $w$, where $j=0,1,\ldots,k-1$. So $v_{3j+1}$ can only receive the color 3, namely $f(v_{3j+1})=3$, $j=0,1,\ldots,k-1$. Since $v_{3j}$ is adjacent to both $w$ and $v_{3j}$ in $G_k$, we have $f(v_{3j})=1$, $j=0,1,\ldots,k-1$. Similarly, we can obtain $f(v_{3j+2})=2$, $j=0,1,\ldots,k-1$. Therefore, the 3-coloring $f$ is uniquely decided as shown in Fig. 2 and then $G_k$ is uniquely 3-colorable.
\end{proof}

\begin{theorem} \label{theorem3.2}
 For any odd $k$ with $k\geq 3$, $G_k$ is edge-critical.
\end{theorem}
\begin{proof}
To complete the proof it suffices to show that $G_k-e$ is not uniquely 3-colorable for any edge $e\in E(G_k)$. Let $f$ be a uniquely 3-coloring of $G_k$ shown in Fig. 2. Denote by $E_{ij}$ the set of edges in $G_k$ whose ends colored by $i$ and $j$, respectively, where $1\leq i< j\leq 3$. Namely
$$E_{ij}=\{xy: xy\in E(G_k), f(x)=i, f(y)=j\}, 1\leq i< j\leq 3.$$

\texttt{Observation 1.} Both the subgraphs $G_k[E_{13}]$ and  $G_k[E_{23}]$ of $G_k$ induced by $E_{13}$ and $E_{23}$ are trees.

\texttt{Observation 2.} The subgraph $G_k[E_{12}]$ of $G_k$ induced by $E_{12}$ consists of $k$ internally disjoint paths $uv_{3i-1}v_{3i}w$, where $i=1,2,\ldots,k$.

If $e\in E_{13}\cup E_{23}$, then $G_k-e$ is not uniquely 3-colorable by Observation 1. Suppose that $e\in E_{12}$. By Observation 2, there exists a number $t\in \{1,2,\ldots,k\}$ such that $e\in \{uv_{3t-1}, v_{3t-1}v_{3t},v_{3t}w\}$. Moreover, $G_k-e$ contains at least one vertex of degree 2. By repeatedly deleting vertices of degree 2 in $G_k-e$ we can obtain a subgraph $G_k-\{v_{3t-1},v_{3t}\}$ of $G_k$. Now we prove that $G_k-\{v_{3t-1},v_{3t}\}$ is not uniquely 3-colorable.

It can be seen that the restriction $f_0$ of $f$ to the vertices of $G_k-\{v_{3t-1},v_{3t}\}$ is a 3-coloring of $G_k-\{v_{3t-1},v_{3t}\}$. On the other hand,  $G_k-\{v_{3t-1},v_{3t},u,w\}$ is a path, denoted by $P$. Let $f'(u)=f'(w)=1$ and alternately color the vertices of $P$ by the other two colors. We can obtain a 3-coloring $f'$ of $G_k-\{v_{3t-1},v_{3t}\}$ which is distinct from $f_0$. Since each 3-coloring of $G_k-\{v_{3t-1},v_{3t}\}$ can be extended to a 3-coloring of $G_k-e$, we know that $G_k-e$ is not uniquely 3-colorable when $e\in E_{12}$.

Since $E(G_k)=E_{12}\cup E_{13}\cup E_{23}$, $G_k-e$ is not uniquely 3-colorable for any edge $e\in E(G_k)$.
\end{proof}

Note that $G_k$ has $3k+2$ vertices and $7k$ edges by the construction. From Theorem \ref{theorem3.2} we can obtain the following result.

\begin{corollary} \label{corollary3.3}
There exist an infinite family of edge-critical uniquely 3-colorable plane graphs with $n$ vertices and $\frac{7}{3}n-\frac{14}{3}$ edges, where $n(\geq 11)$ is odd and $n\equiv 2\pmod{3}$.
\end{corollary}

Denote by $size(n)$ the upper bound of the number of edges of edge-critical uniquely $3$-colorable planar graphs with $n$ vertices. Then by Corollary \ref{corollary3.3} and the result due to Matsumoto \cite{Matsumoto2013}, we can obtain the following result.

\begin{corollary} \label{corollary3.4}
 For any odd integer $n$ such that $n\equiv 2\pmod{3}$ and $n\geq 11$, we have $\frac{7}{3}n-\frac{14}{3}\leq size(n) \leq\frac{8}{3}n-\frac{17}{3}$.
\end{corollary}

\section{Concluding remarks}
In this paper we obtained a structural property of uniquely 3-colorable plane graphs. We proved that every uniquely 3-colorable plane graph has adjacent $(3,k)$-faces, where $k\leq 5$, and the bound 5 for $k$ is best possible. Fig. \ref{figure1} shows a uniquely 3-colorable plane graph having neither adjacent $(3,3)$-faces nor adjacent $(3,4)$-faces. But this plane graph is 2-connected. This prompts us to propose the following conjecture.

\begin{conjecture}\label{conjecture1}
Let $G$ be a 3-connected uniquely 3-colorable plane graph. Then $G$ has adjacent $(3,k)$-faces, where $k\leq 4$.
\end{conjecture}

It can be seen that the uniquely 3-colorable plane graph $G_k$ constructed in Section 3 is 3-connected. So if Conjecture \ref{conjecture1} is true, then the bound 4 for $k$ is best possible.


\begin{thebibliography}{110}

{\small

\bibitem{Aksionov1977} V. A. Aksionov, On uniquely 3-colorable planar graphs, Discrete Math. 20 (1977) 209-216.

\bibitem{Birkhoff1912} G. D. Birkhoff, A determinant formula for the number of ways of colouring a map, Chromatic polynomials,
Ann. of Math. 14 (1912) 42-46.

\bibitem{Bollob¨¢s1978} B. Bollob\'{a}s, Uniquely colorable graphs, J. Combin. Theory, Ser. B, 25 (1) (1978) 54-61.

\bibitem{Bondy2008} J. A. Bondy, U. S. R. Murty, Graph Theory, Springer, 2008.

\bibitem{Chartrand1969} G. Chartrand, D. P. Geller, On uniquely colorable planar graphs, J. Combin. Theory, 6 (3) (1969) 271-278.

\bibitem{Harary1968} F. Harary, D. Cartwright, On the coloring of signed graphs, Elem. Math. 23 (1968) 85-89.

\bibitem{Harary1969} F. Harary, S. T. Hedetniemi, R. W. Robinson, Uniquely colorable graphs, J. Combin. Theory, 6 (3) (1969) 264-270.


\bibitem{Matsumoto2013} N. Matsumoto, The size of edge-critical uniquely 3-colorable planar graphs, Electron. J. Combin. 20 (3) (2013) $\#$P49.

\bibitem{Mel'nikov1977} L. S. Mel'nikov, R. Steinberg, One counterexample for two conjectures on three coloring, Discrete Math. 20 (1977) 203-206.

\bibitem{Read1968} R. C. Read, An introduction to chromatic polynomials, J. Combin. Theory
4 (1968), 52-71.
\bibitem{Whitney1932} H. Whitney, The coloring of graphs, Ann. of Math. 33 (2)  (1932) 688-718.

}

\end{thebibliography}
\end{document}